\documentclass[12pt]{article}
\usepackage{amssymb,amsmath,amsfonts}
\usepackage{amsthm}
\usepackage{color}

\numberwithin{equation}{section} \theoremstyle{plain}
\newtheorem{theorem}{Theorem}[section]
\newtheorem{corollary}[theorem]{Corollary}

\theoremstyle{definition}
\newtheorem{definition}{Definition}[section]
\newtheorem{example}{Example}[section]
\theoremstyle{remark}
\newtheorem{remark}{\rm\bf Remark}[section]

\allowdisplaybreaks

\begin{document}

\title{Discrete Stieltjes classes for log-Heine type distributions}
\author{Sofiya Ostrovska and Mehmet Turan}
\date{}
\maketitle

\begin{center}
{\it Atilim University, Department of Mathematics,  Incek  06836, Ankara, Turkey}\\
{\it e-mail: sofia.ostrovska@atilim.edu.tr, mehmet.turan@atilim.edu.tr}\\
{\it Tel: +90 312 586 8211,  Fax: +90 312 586 8091}
\end{center}

\begin{abstract}
The Stieltjes classes play a significant  role  in the moment problem since they permit to expose explicitly an infinite family of probability distributions all having equal moments of all orders.  Mostly, the Stieltjes classes have been considered for absolutely continuous distributions. In this work, they have been considered for discrete distributions. New results on their existence in the discrete case are presented. Examples for some widely used discrete distributions are provided.

\end{abstract}



\section{Introduction }

 Stieltjes classes actually appeared in \cite{stieltjes}, while the name may be viewed as present-day.
For good reasons, J. Stoyanov \cite{jap} suggested to use the name `Stieltjes classes'
and  triggered their systematic study, which is still in progress. See, for example  \cite{ recent, pakes, counter}.
Mostly, Stieltjes classes have been considered for absolutely continuous distributions as classes of different probability densities with the same sequence of moments.  However, they may also be used to construct  sets  of discrete distributions with the same sequence of moments.

For the sequel, we need the following definitions.

\begin{definition} Let $X$ be a random variable possessing a discrete distribution with probability mass function $p_X=p$
given by $p(x_j)=p_j,\; j\in \mathbb{N}_0.$ A sequence $h=\{h_j\}_{j\in \mathbb{N}_0}$ is a \textit{perturbation} for $p$ if
$M_h:=\displaystyle\sup_j |h_j|=1$ and
\begin{align*}
\sum_{j=0}^\infty x_j^k p_jh_j=0 \quad  \text{for all } k\in{\mathbb N}_0.
\end{align*}
\end{definition}

It has to be noticed that not all probability mass functions own perturbations. Clearly, the existence of perturbation is stipulated by the moment indeterminacy of the underlying discrete distribution. In this connection, the next definition can be formulated.
\begin{definition} Given a probability mass function $p$ and its perturbation $h,$
the set
\begin{align*}
{\bf S}:=\{g:  g = p(1+\varepsilon h), \ \varepsilon \in[-1, 1]\}
\end{align*}
is called a (discrete) {\it Stieltjes class} for $p$ generated by $h.$
\end{definition}

Examples of discrete Stieltjes classes are provided in \cite{berg} and \cite[Section 11]{counter}. It has to be pointed out that, since \cite{berg} had appeared before \cite{jap} was published, the name `Stieltjes class' had not been used there. However, the results of \cite{berg} can be easily restated in terms of the Stieltjes classes. It is worth  mentioning  that C. Berg characterizes discrete moment-indeterminate distributions possessing Stieltjes classes as the ones which are not extreme points in the set of distributions with the same moment sequences. See \cite[Proposition 1.1]{berg}. 

In the present paper,  the existence of Stieltjes classes related to certain families of discrete distributions is investigated. More precisely, given a non-negative integer-valued random variable $X,$ this study aims to examine the presence of Stieltjes classes for the probability mass function of $Y=a^X,$ $a>0,$ $a\neq 1.$ Obviously, when $a\in (0,1),$ there are no Stieltjes classes for $Y$ because  the distribution of $Y$ is moment-determinate as it has a bounded support. Therefore, only the case $a>1$ will be considered.  Denote by 
\begin{equation*}p_j=\mathbf{P}\{X=j\},\quad j\in\mathbb{N}_0.\end{equation*} Correspondingly, the probability generating function of $X$ can be written as:
\begin{equation*}\label{entire} f(z)=\sum_{j=0}^\infty p_jz^j,\quad z\in\mathbb{C}.\end{equation*} 
Clearly, $Y$ has finite moments of  all orders if and only if $f(z)$ is entire. Some conditions in terms of coefficients and growth estimates of $f(z)$ for the Stieltjes classes of $P_Y$ to exist are established. It will be proved that,
under the condition
\begin{align*}\label{cond1}
p_j\geqslant C a^{-j(j-1)/2} \quad\mathrm{for \;all\;}j\in \mathbb{N}_0\quad \mathrm{and\;some\; } C>0,
\end{align*}
the probability mass function  of $Y$ has a perturbation and, as a result, the distribution  is moment-indeterminate. On the other hand, if
$$p_j=o\left(a^{-j(j-1)/2}\right) \quad \text{as } \ j\to \infty,$$
 no perturbation function exists.  
The application of these results to  the case when  $X$ has a log-concave  distribution is provided.

Finally, as a model example, a random variable $X$ possessing Heine distribution is considered. It is one of the important $q$-distributions, see \cite[Section 2.3]{char}.
Its probability mass function  involves the following $q$-analogue of the exponential function:
\begin{align*}
e_q(t)=\prod_{j=0}^\infty \left(1-t(1-q)q^j\right)^{-1},\quad 0<q<1, \quad t<1/(1-q).
\end{align*}
It is said that a random variable $X$ has Heine distribution with parameter $\lambda >0$ - and is written $X\sim Heine(\lambda)$ - if its probability mass function is given 
by:
\begin{equation}\label{heine} 
p_j=e_q(-\lambda)\frac{q^{j(j-1)/2}\lambda^j}{[j]_q!},\quad j\in\mathbb{N}_0,\; \lambda >0,\;  0<q<1.
\end{equation}
Here, $[j]_q!$ is the $q$-factorial  defined as:
$$[0]_q!:=1,\;[j]_q!:= \frac{(q;q)_j}{(1-q)^j}\quad \text{with}\quad (q;q)_j=\displaystyle\prod_{s=1}^j (1-q^s).$$ 
Conventionally, 
$$(q;q)_0=1 \quad \text{and} \quad (q;q)_\infty=\prod_{s=1}^\infty (1-q^s).$$  This distribution is viewed as a $q$-analogue of the  Poisson one since when $q\rightarrow 1^-,$  the Poisson distribution with parameter $\lambda$ is recovered. If $X\sim Heine(\lambda),$ the distribution of $Y=a^X, a>0, a\neq 1$ is called a \textit{log-Heine} distribution.
It is not difficult to see that the Heine distribution possesses finite moments of all orders, and, moreover, it is moment-determinate as its moment-generating function exists for all real numbers. The same is true for log-Heine distribution with $a\in(0, 1)$ since - as mentioned previously - it has a bounded support. The investigation of moment-(in)determinacy of the log-Heine distribution in the case $a>1$ is not that simple. In Example \ref{logh}, it is shown that a Stieltjes class for $Y$ exists if and only if $a>1/q$ or $a=1/q$ and $\lambda(1-q) \geqslant 1.$

Last but not least, it has to be acknowledged that this study is motivated by Examples 11.7 and 11.8 of \cite{counter}.

\section{Results and examples}

In the sequel, the letter $C$ - with or without indexes - denotes a positive constant whose value is of no concern. Note that the same letter may be assigned to denote constants with different numerical values.

\begin{theorem}\label{th1}
Let a random variable $X$ have  probability mass function $p_X(j)=p_j,$ $j\in\mathbb{N}_0,$ and $Y=a^X,\; a>1.$ If
\begin{equation}\label{W}
p_j\geqslant C a^{-j(j-1)/2} \quad {for \;all}\;j\geqslant 0,
\end{equation}
then a perturbation of $p_Y$ exists and, therefore, the distribution of $Y$ is moment-indeterminate, provided that $Y$ has the moments of all orders.
\end{theorem}
\begin{proof}
To establish this result, it suffices to find a bounded non-zero sequence $\tilde{h}=\{\tilde{h}_j\}_{j\geqslant 0}$ satisfying
\begin{equation}\label{ph}\sum_{ j\geqslant 0} a^{kj}p_j\tilde{h}_j=0\;\;\mathrm{for\quad all\quad}k\in\mathbb{N}_0.
\end{equation}
Recall the following identity  established by Euler - see, for example, \cite[formula (1.23)]{char}:
\begin{align}\label{eul}
 \prod_{j=0}^\infty (1+q^j t) =
\sum_{j=0}^\infty \frac{q^{j(j-1)/2}}{(q;q)_j}\, t^j,\quad 0<q<1,\quad t\in \mathbb{C}.
\end{align}
Set 
\begin{align}\label{h1}
\tilde{h}_j =
 \frac{(-1)^ja^{-j(j-1)/2}}{(1/a;1/a)_j p_j}, \quad j \in \mathbb{N}_0.
\end{align}
Obviously, $\tilde{h}\neq 0,$ and owing to  \eqref{W},
\begin{equation*}
|\tilde{h}_j|\leqslant\frac{a^{-j(j-1)/2}}{(1/a;1/a)_j Ca^{-j(j-1)/2}}\leqslant \frac{1}{C(1/a;1/a)_\infty}=:C_1.
\end{equation*}
The validity of \eqref{ph} follows immediately from \eqref{eul} because for 
$$\varphi(t):=\prod_{s=0}^\infty \left(1-\frac{t}{a^s}\right)$$
one has 
$$0=\varphi(a^k)= \sum_{j=0}^\infty a^{kj} \frac{(-1)^ja^{-j(j-1)/2}}{(1/a;1/a)_j}=\sum_{j=0}^\infty a^{kj}p_j\tilde{h}_j \quad \text{for all} \quad k\in{\mathbb N}_0.$$
Taking $h:=\tilde{h}/M_{\tilde{h}}, $ where $M_{\tilde{h}}=\sup_j|h_j|$, one obtains a perturbation $h$ for $p_Y$ and in this way completes the proof.
\end{proof}

\begin{corollary}
Let $X$ and $Y$ be as in Theorem \ref{th1}. Then, the set 
\begin{align*}
{\bf S}=\left\{g=\{g_j\}: g_j =p_j(1 +\varepsilon h_j), \varepsilon\in[-1, 1]\right\},
\end{align*} where $h$ is constructed by means of \eqref{h1}, forms  a Stieltjes class for $p_Y.$
\end{corollary}

\begin{example} (Log-Poisson distribution) Let $X$ have Poisson distribution with parameter $\lambda$.
Then, the probabilities $p_j= {\lambda^j e^{-\lambda}}/{j!},$ $j\in\mathbb{N}_0$ satisfy condition \eqref{W} for every $a>1.$ Indeed, with the help of Stirling's formula, it can be observed that
\begin{equation*}  p_j\sim C\exp\{j\ln\lambda-(j+1/2)\ln j+j\} \geqslant C_1\exp\left\{-\frac{j(j-1)}{2}\ln a\right\}.
\end{equation*}
Hence, by Theorem \ref{th1}, the distribution of
  $Y=a^X,$ $a>1,$  is moment-indeterminate and a Stieltjes class for $p_Y$ can be written in the form:
\begin{align*}
{\bf S}:=\{g:  g = p_Y(1+\varepsilon h), \ \varepsilon \in[-1, 1]\},
\end{align*} where $ h=\tilde{h}/M_{\tilde{h}}$ and $\displaystyle \tilde{h}_j=\frac{(-1)^j e^{\lambda}j!\,a^{-j(j-1)/2}}{\lambda^j (1/a;1/a)_j},\quad j\in\mathbb{N}_0.$
 \end{example}
The sharpness of Theorem \ref{th1} can be demonstrated by the next result. The  proof below is similar to the one presented in 
Lemma 2.6 of \cite{qarxiv}.

\begin{theorem}\label{th2}
Let $X$ and  $Y$ be as in Theorem \ref{th1}. If
\begin{align*}
p_j = o(a^{-j(j-1)/2}) \quad \text{as} \quad j \to +\infty,
\end{align*}
then there are no perturbation functions for $p_Y.$
\end{theorem}

\begin{proof} To begin with, recall the standard notation $M(r;f):= \max_{|z|=r} |f(z)|$ where $f(z)$ is analytic in $\{z: |z| \leqslant r\}.$ 
Assume that $0\neq h=\{h_j\}_{j\geqslant 0}$ is a bounded sequence such that
\begin{equation*}
\sum_{j=0}^\infty a^{kj}p_jh_j=0\quad\mathrm{for\;\;all}\quad k\in \mathbb{N}_0.
\end{equation*}
Set $c_j:=p_jh_j.$  Obviously, the coefficients $c_j$ satisfy $c_j=o\left(a^{-j(j-1)/2}\right)$ as $j\rightarrow +\infty$,   implying  that the entire function
 $\phi(z)=\sum_{j=0}^{\infty} c_j z^j $  enjoys the estimate:
\begin{equation}\label{grow}M(r;\phi)\leqslant \sum_{j=0}^\infty|c_j|r^j =o\left(\exp\left\{\frac{\ln^2 r}{2\ln a}+\frac{\ln r}{2}\right\}\right) \quad \text{as} \quad r\to\infty\end{equation} because 
$$\sum_{j= 0}^\infty a^{-j(j-1)/2}r^j\leqslant C \exp\left\{\frac{\ln^2 r}{2\ln a}+\frac{\ln r}{2}\right\},\quad C=C(a)\;\text{and} \;r\geqslant r_0.$$
On the other hand, since $\phi(a^k)=0$ for all $k\in \mathbb{N}_0,$ Jensen's Theorem \cite[\S 3.61]{titch} implies that, for $k\in \mathbb{N}_0,$ the inequality below is valid:
\begin{align*}
M(r;\phi) \geqslant C \exp\left\{\frac{\ln^2 r}{2\ln a}+\frac{\ln r}{2}\right\}\quad\text{when}\quad r=a^k,
\end{align*}
which, however,  contradicts  \eqref{grow}. The proof is complete.
\end{proof}
 The  next  statement deals with discrete log-concave distributions. It is known that a discrete distribution is \textit{log-concave} if
\begin{equation}\label{logconc}
p_j^2\geqslant p_{j-1}p_{j+1},\quad j\in \mathbb{N}.
\end{equation} 
Such distributions have been studied by many authors from different angles and have shown to be of interest for applications. See, for example, \cite{logconcave} and \cite{sw}.

 \medskip
 \begin{theorem}Let $X$ have a log-concave  distribution, whose probability generating function $f(z)$ is entire and satisfies:
 \begin{align}\label{entpgf}0<\lim_{r\rightarrow\infty}\frac{\ln f(r)}{\ln^2 r}=:\beta <\infty.\end{align}
Then, for $a>\exp\{1/(2\beta)\}$ the distribution of $Y=a^X$ has Stieltjes classes, while for $a<\exp\{1/(2\beta)\}$ there are no Stieltjes classes. 
In the case $a=\exp\{1/(2\beta)\},$ additional information is needed.
\end{theorem}
\begin{proof} 
Since $f$ has positive coefficients, one has $M(r; f)=f(r)$ and the Cauchy estimates for the coefficients of $f$ imply that 
$p_j\leqslant f(r)r^{-j}.$ By virtue of \eqref{entpgf}, one has
$$\forall \varepsilon>0, \quad \ln p_j\leqslant \ln f(r)-j\ln r\leqslant (\beta+\varepsilon)\ln^2r-j\ln r, \quad r\geqslant r_0=r_0(\varepsilon).$$   
Taking the minimum with respect to $r,$ one obtains 
$\ln p_j\leqslant -\frac{j^2}{4(\beta+ \varepsilon)}$ for  $j\geqslant j_0(\varepsilon).$   Hence, 
\begin{align}\label{psup}\limsup_{j\rightarrow\infty}\frac{\ln p_j}{j^2}\leqslant -\frac{1}{4\beta}.\end{align} 
Notice that \eqref{psup} does not need the log-concavity of the sequence $\{p_j\}$ and also is valid for the upper limit rather than limit in the definition of $\beta.$  The latter estimate implies that if  $a<\exp\{1/(2\beta)\}$, then the upper estimate for $p_j$ gives  $p_j=o\left(a^{-j(j-1)/2}\right),\;j\rightarrow \infty,$ and Theorem \ref{th2} is applicable. 

It is commonly known that, in general, it is impossible to obtain lower estimates for the coefficients of an entire function from the function's growth estimate. Below, such estimates will be derived from 
 the  result  by V. Boicuk and A. Eremenko on the Dirichlet series (\cite[Theorem 3]{boer}) under the additional condition that $\{p_j\}$ is log-concave. Below, we present a version of their proof showing that if the probability generating function $f(z)=\sum_{j=0}^\infty p_jz^j$ satisfies the conditions of the theorem, then 
$\ln p_j\geqslant -j^2/(4\beta) + \alpha_j,$ where $\alpha_j=o(j^2)$ as $j\rightarrow\infty.$

Assume that, there exists a subsequence $\{p_k\},\;k\in K\subset {\mathbb N}_0$ such that 
$$\ln p_k\leqslant -\frac{k^2}{4\gamma},\quad 0<\gamma <\beta,\;k\in K.$$
It follows from \eqref{logconc} that each term $p_kr^k$ is the maximum term for some $r=r_k,$ see, for example, \cite[Part IV, Ch. 5, problem 43]{polya}. That is, $\mu(r_k)=\max_j p_jr_k^j=p_kr_k^k.$ Then, for every $\delta>1$ and $k$ large enough, $M(r_k;f)\leqslant p_k (\delta r_k)^k$
as $M(r_k)/\mu (r_k)\rightarrow 1,$ $k\rightarrow\infty.$ We refer to \cite[Part IV, Ch. 5, problem 54]{polya}. For $k\in K,$ this implies that 
\begin{align*}
\ln M(r_k;f) &\leqslant  \ln p_k+k\ln (\delta r_k)\leqslant -\frac{k^2}{4\gamma}+k\ln (\delta r_k)\\
&\leqslant \max_{t}\left( -\frac{t^2}{4\gamma}+t\ln (\delta r_k)\right)=\gamma \ln^2(\delta r_k), \quad k\geqslant k_0.
\end{align*}
Hence, $$\limsup_{k\rightarrow\infty}\frac{\ln M(r_k;f)}{\ln^2 r_k} \leqslant \limsup_{k\rightarrow\infty} \frac{\gamma\ln^2 (\delta r_k)}{\ln^2 r_k}=\gamma<\beta,$$contrary to \eqref{entpgf}. The contradiction shows that \begin{equation}\label{pinf}\liminf_{j\rightarrow\infty}\frac{\ln p_j}{j^2}\geqslant -1/(4\beta),\end{equation} whence  \begin{equation}\label{pj}p_j\geqslant \exp \{-j^2/(4\beta)+o(j^2)\},\quad j\rightarrow \infty.\end{equation}  Since the sequence $\{p_j\}$ is log-concave and not ultimately 0, it follows that all $p_j>0,$ and hence \eqref{pj} holds for all $j\in \mathbb{N}_0. $
This implies immediately that when $a>\exp\{-1/(2\beta)\},$ there holds
$$p_j\geqslant  Ca^{-j(j-1)/2}\quad \text{for some} \quad C>0 \quad\text{and all}\quad j\in \mathbb{N}_0.$$ By Theorem \ref{th1}, in this case the Stieltjes classes for the probability mass function of $Y=a^X$ exist. 

 Finally, it has to be pointed out that if $p_j=C\exp\{-j(j-1)/(4\beta)\},$ then Stieltjes classes exist by Theorem \ref{th1}, while for $p_j=C\exp\{-j(j+1)/(4\beta)\},$ there are no Stieltjes classes due to Theorem \ref{th2}. Notice that in both cases the distribution of $X$ is log-concave.
\end{proof}

\begin{remark} By juxtaposing \eqref{psup}  and \eqref{pinf}, it can be derived that, if $\{p_j\}$ is log-concafe and $f(z)$ satisfies \eqref{entpgf}, then $$\lim_{j\rightarrow\infty}\frac{\ln p_j}{j^2}= -\frac{1}{4\beta}.$$

\end{remark}

\begin{example} \label{logh} (Log-Heine distribution) Let $X\sim Heine(\lambda),$ that is, by \eqref{heine},
$$p_j=e_q(-\lambda)\frac{q^{j(j-1)/2}(\lambda(1-q))^j}{(q,q_j},\quad j\in\mathbb{N}_0.$$ Taking into account that $(q;q)_\infty<(q;q)_j<1,$ one obtains that 
$$C_1(qa)^{j(j-1)/2}[\lambda(1-q)]^j\leqslant \frac{p_j}{a^{-j(j-1)/2}}\leqslant C_2(qa)^{j(j-1)/2}[\lambda(1-q)]^j.$$ The last estimate implies that the Stieltjes classes for the log-Heine distribution exist if and only if either $a>1/q$ or $a=1/q$ and $\lambda(1-q)\geqslant 1.$
\end{example}

\section{Acknowledgements} The authors extend their appreciations to Professor Alexandre Eremenko (Purdue University, USA) for his valuable comments.

\end{document}